\documentclass[11pt]{amsproc}
\usepackage{mathrsfs}
\usepackage{mathrsfs, color}
\usepackage{stmaryrd}
\usepackage{cases}
\usepackage{amssymb}
\usepackage{amsmath}
\usepackage{amsfonts}
\usepackage{graphicx}
\usepackage{amsmath,amstext,amsbsy,amssymb}

\newtheorem{theorem}{Theorem}[section]
\newtheorem{lemma}[theorem]{Lemma}

\newtheorem{proposition}[theorem]{Proposition}
\newtheorem{corollary}[theorem]{Corollary}

\theoremstyle{definition}

\newtheorem{example}[theorem]{Example}

\theoremstyle{remark}
\newtheorem{remark}[theorem]{Remark}

\numberwithin{equation}{section} \errorcontextlines=0

\newcommand{\al}{\alpha}

\newcommand{\Det}{\mathrm{Det}}
\newcommand{\Pf}{\mathrm{Pf}}
\newcommand{\hdet}{\mathrm{Det}}
\newcommand{\GL}{\mathrm{GL}}
\newcommand{\Mat}{\mathrm{Mat}}

\begin{document}

\title[Quantum Hyperdeterminants and Hyper-Pfaffians]
{Quantum Hyperdeterminants and Hyper-Pfaffians}
\author{Naihuan Jing}
\address{
Department of Mathematics,
   North Carolina State University,
   Raleigh, NC 27695, USA and
   Department of Mathematics, Shanghai University, Shanghai 200444, China}
\email{jing@math.ncsu.edu}
\author{Jian Zhang}
\address{Institute of Mathematics and Statistics, University of Sao Paulo, Sao Paulo, Brazil 05315-970}
\email{zhang@usp.ime.br}
\thanks{*Corresponding author: Jian Zhang}

\begin{abstract}
The notion of generalized quantum monoids is introduced.
It is proved that the quantum coordinate ring of the monoid can be lifted to
a quantum hyper-algebra, in which the quantum determinant and
quantum Pfaffian
are sent to the quantum hyperdeterminant and quantum hyper-Pfaffian respectively.
The quantum hyperdeterminant in even dimension
is shown to be a $q$-analog of Cayley's first hyperdeterminant.
\end{abstract}
\maketitle
\section{Introduction}

In mathematics and physics, one is often lead to consider $m$-dimensional hypermatrices
$A=(a_{i_1\ldots i_{m}})$ indexed by multi-indices, while the usual rectangular matrices are $2$-dimensional
\cite{So1, So2, GKZ, HT, Mat}.
An $m$-dimensional hypermatrix is said to have the format $n^m=n\times\cdots\times n$ if each
basic index $i_k$ runs through $1, \ldots, n$.

The space $\mathscr A_1$ of $m$-dimensional hypermatrices forms a representation of the group $\GL_n^{\otimes m}\times
\GL_n^{\otimes m}$ via the left and right action:

\begin{equation}
\GL_n^{\otimes m}\times\mathscr A_1 \times  \GL_n^{\otimes m}\longrightarrow \mathscr A_1.
\end{equation}
Here the left action is defined by the multiplication rule ${\mathrm{M}}_n\times \mathscr A_1\overset{\circ_k}{\longrightarrow} \mathscr A_1$:
\begin{equation}\label{e:action}
(B\circ_k A)_{i_1\ldots i_m}=\sum_{j=1}^n b_{i_k,j}a_{i_1\ldots i_{k-1}ji_{k+1}\ldots i_m},
\end{equation}
and the right action is defined by taking the transpose of \eqref{e:action}.

Cayley had laid the foundation of classical invariant theory and
multilinear algebra, in which he introduced the notion of hyperdeterminants for studying hypermatrices and
tensors, see \cite{GKZ} for a modern account of Cayley's theory.
Cayley's first hyperdeterminant $\hdet$ \cite{C, L1, L2} is defined for any even-dimensional hypermatrix
$A=(a_{i_1\ldots i_{2m}})$, where $1\leq i_{j}\leq n$ and $a_{i_1\ldots i_{2m}}$'s mutually commute:
\begin{equation}\label{e:hdet0}
{\Det}(A)=\frac{1}{n!}
\sum_{\sigma_{1},\ldots,\sigma_{2m}\in S_{n}}(-1)^{\sum_{i=1}^{m}\ell(\sigma_i)}\prod_{i=1}^{n}
a_{\sigma_1(i),\sigma_2(i),\ldots,\sigma_{2m}(i)}.
\end{equation}

One remarkable property of $\hdet$ is the relative invariance under the action of $\GL_{n}^{\otimes 2m}$:
\begin{equation}\label{e:inv}
\hdet(A\circ_k B)=\hdet(B\circ_k A)=\hdet(A)\det(B),
\end{equation}
where $B$ is any $n\times n$-matrix. More generally, for fixed $k, l$, one can define the following product for
two hypermatrices of suitable formats:
\begin{equation}\label{e:product}
(A\,_k\circ_l B)_{i_1\ldots i_{m+n-2}}=\sum_{j} a_{i_1\ldots i_{k-1}ji_{k+1}\ldots i_m}b_{i_{m+1}\ldots i_{l-1}ji_{l+1}\ldots i_{m+n-2}}.
\end{equation}
Then one has the invariance property:
\begin{equation}\label{e:inv2}
\hdet(A\,_k\circ_l B)=\hdet(A)\hdet(B),
\end{equation}

Our first aim in the present work
is to introduce the notion of quantum hypermatrices and generalize the algebra $\Mat_q(n)$
of quantum matrices to that of quantum hypermatrices. We will construct a quantum deformation of
Cayley's hyperdeterminant for any hypermatrix $A=(a_{i_1\ldots i_{m}})$, where $a_{i_1\ldots i_{m}}$ are
not necessarily commutative in general, and show that the new determinant satisfies a $q$-deformation of identity \eqref{e:inv}. In \cite{Ma1, Ma2},
Manin showed that the quantum matrix ring $\Mat_q(n)$ can be formulated as
the quantum transformations that preserve the quantum exterior algebra and
quantum Weyl algebra.

To generalize Manin's idea we first show that the quantum Weyl algebra
can be replaced by the quantum exterior algebra provided that
one imposes the invariance for the dual quantum transformation $A^T$.
This idea then works well for quantum hypermatrices when we require that all
matrix realignments are transformed according to the rule of the quantum exterior algebra.
We show that the transformation rules provide enough quantum symmetry to warrant
that the hyperdeterminant is a quantum volume element.

The second aim is to generalize the notion of Pfaffians to higher dimensional cases for the
quantum hypermatrices.
We will introduce the notion of anti-symmetric hypermatrices and quantum hyper-Pfaffian, and
prove that the quantum hyper-Pfaffian is given by a special volume element of the
quantum 2-form.

One important property and new feature of our quantum hyperdeterminants is that we are able to define them for
any dimensional hypermatrices, removing the restriction that Cayley's hyperdeterminant
is only defined for even dimension (cf. \cite{L1, L2}).
Even more interesting is the fact that our hyperdeterminant also works for $q=1$ for
noncommutative Manin-type hypermatrices. These hyperdeterminants at odd dimension will only vanish when
the matrix elements are commutative. This is in complete agreement with the general phenomenon
that classical singularity is often better regularized at the level of quantum deformation.

We remark that our guiding principle is to try to find minimum defining relations
to ensure both quantum hyperdeterminants and quantum hyper-Pfaffians work. In other
words, the relations that we have found will ensure that
the diamond lemma is satisfied to rearrange the products in the determinant and Pfaffians.

The quantum hyper-Pfaffian generalizes the quantum Pfaffian studied in \cite{JZ}, which can be viewed as
the quantum Pfaffian for matrices, while our current quantum Pfaffians are defined for hyper-matrices.
Some of the classical situation have been considered in \cite{Re}.
We stress that the current version of quantum hyper-Pfaffians is different from that of the quantum hyper-Pfaffian
in \cite{JZ}, which was defined in terms of higher degree $q$-forms. The current quantum hyper-Pfaffian is defined
as the Pfaffian element for quantum hyper-matrices.

\section{Quantum monoids}

\subsection{Quantum exterior algebras}
Let $q$ be a non-zero complex number.
 The quantum exterior algebra $\Lambda_n=\Lambda$ is the quadratic algebra  generated by $x_1, \ldots, x_n$ over the field $F$ subject to the
following relations:
\begin{align}\label{qwedge1}
&x_j \wedge x_i=-qx_i\wedge x_j, \\
&x_i\wedge x_i=0, \label{qwedge2}
\end{align}
where $i<j$. The algebra $\Lambda$ is naturally $\mathbb Z_{n+1}$-graded and decomposes itself:
\begin{equation}
\Lambda=\bigoplus_{k=0}^n\Lambda^k,
\end{equation}
where the $k$th homogeneous subspace $\Lambda^k$ is spanned by $x_{i_1}\wedge\cdots\wedge x_{i_k}$,
$1\leq i_1<\cdots<i_k\leq n$. So $dim(\Lambda)=2^n$.

The quantum monoid  $\Mat_q(n)$ is the unital bialgebra generated by $a_{ij}$ subject to the following
quadratic relations:
\begin{align}\label{relation a1}
&a_{ik}a_{il}=qa_{il}a_{ik}, \\\label{relation a2}
&a_{ik}a_{jk}=qa_{jk}a_{ik}, \\\label{relation a3}
&a_{jk}a_{il}=a_{il}a_{jk}, \\\label{relation a4}
&a_{ik}a_{jl}-a_{jl}a_{ik}=(q-q^{-1})a_{il}a_{jk},
\end{align}
where $i<j$ and $k<l$. The coproduct is given by
\begin{align}
\Delta(a_{ij})=\sum_ka_{ik}\otimes a_{kj}.
\end{align}

One can use the quantum exterior algebra to give a
conceptual presentation of the quadratic relations in the quantum monoid.
We  consider 
the tensor product $\Mat_q(n)\otimes \Lambda_n$, where we require that
$x_k$'s and $a_{ij}$'s commute. We will simply write $ax$ for the tensor product $a\otimes x\in A\otimes \Lambda$, and the wedge products are similarly written as follows.
\begin{equation*}
(ax)\wedge (by)=ab(x\wedge y).
\end{equation*}
Thus $x_ka_{ij}=a_{ij}x_k$ for any $1\leq i, j, k\leq n$.
The following is a reformulation of Manin's result \cite{Ma1, Ma2}.

\begin{proposition}\label{P:quadr} The defining relations of the generators
$a_{ij}$ for the quantum monoid $\Mat_q(n)$ are equivalent to the following rule:
suppose that $x_i$'s obey the relations (\ref{qwedge1})-(\ref{qwedge2}), then $y_i=\sum_ka_{ik}x_k$ and
 $y'_i=\sum_ka_{ki}x_k$ both satisfy the relations (\ref{qwedge1})-(\ref{qwedge2}).
\end{proposition}


\subsection{Generalized quantum monoids} We now generalize $\Mat_q(n)$ to higher dimensional quantum hypermatrices.
For the natural number $n$ we denote the index set $\{1, \ldots, n\}$ by $[1, n]$.
Fix $m, n\geq 1$. Let $A=(a_{i_1\ldots i_{m}})$ be a hypermatrix of format $n\times\cdots \times n=n^{m}$.
We will use the matrix realignment to simplify our presentation. For each hypermatrix ${A}$ of format $n^{m}$, one associates
$m$ realigned or folded rectangular matrices ${A}^{(1)}, \ldots, {A}^{(m)}$.
The $k$th matrix realignment ${A}^{(k)}=(a_{i\alpha}^{(k)})$
is a $n\times n^{m-1}$-rectangular matrix with entries
\begin{align}
a_{i\alpha}^{(k)}=a_{j_1\ldots j_{k-1} ij_{k+1}\ldots j_{m}}, \qquad \alpha=(j_1\ldots j_{k-1} j_{k+1}\ldots j_{m}).
\end{align}
Here, $i\in [1, n]$ and $\alpha$ runs through the set $[1, n]^{m-1}$ in the lexicographic order.
For example, the $2\times2\times2\times2$ hypermatrix 
$A$ has the following foldings with respect to the second and third indices:
\begin{equation}
{A}^{(2)}=
\begin{pmatrix}
a_{1111} & a_{1112} & a_{1121} & a_{1122}  & a_{2111} & a_{2112} & a_{2121} & a_{2122}\\
a_{1211} & a_{1212} & a_{1221} & a_{1222}  & a_{2211} & a_{2212} & a_{2221} & a_{2222}\\
\end{pmatrix},
\end{equation}

\begin{equation}
 {A}^{(3)}=
\begin{pmatrix}
a_{1111} & a_{1112} & a_{1211} & a_{1212}  & a_{2111} & a_{2112} & a_{2211} & a_{2212}\\
a_{1121} & a_{1122} & a_{1221} & a_{1222}  & a_{2121} & a_{2122} & a_{2221} & a_{2222}\\
\end{pmatrix}.
\end{equation}

We consider the $n^{m-1}$-dimensional column vector $X=[x_{\alpha}]$ with entries $x_{\alpha}=x_{\al_1}\otimes x_{\al_2}\otimes\cdots\cdots\otimes x_{\al_{m-1}}\in(\Lambda^{1})^{\otimes m-1}$,
$\al_k\in[1, n]$. Here, the entries of $X$ are also ordered in the lexicographic order. For example, if $ m=n=2$,
$$X=(x_{111}, x_{112},x_{121},x_{122},x_{211}, x_{212},x_{221},x_{222})^t.$$

In the following, we assume $q^2\neq-1$. Suppose all $a_{i_1,i_2,\ldots,i_{m}}$ commute with the $x_i$. Define $x^{(k)}_i$'s by
\begin{equation}
\begin{pmatrix}
x_1^{(k)}\\
x_2^{(k)}\\
\cdots\\
x_n^{(k)}\\
\end{pmatrix}
={A}^{(k)}X.
\end{equation}
In other words, $x_i^{(k)}=\sum_{\alpha}a_{i\alpha}^{(k)}x_{\alpha}$, where $\alpha$ runs through the set $[1, n]^{m-1}$.
For any subset $J=\{j_1, \ldots, j_r\}\subset [1, n]$ we denote
$x_{J}=x_{j_1}\wedge\cdots \wedge x_{j_r}$.
For $n\in \mathbb N$, we define the $q$-number $[n]_q$ by
\begin{equation}\label{e:qnum}
[n]_q=1+q+\cdots +q^{n-1}
\end{equation}
and the quantum factorial $[n]_q!=[1]_q[2]_q\cdots [n]_q$.

We remark that $x_{\al}$ is an element from the tensor product
$(\Lambda^1)^{\otimes(m-1)}$, while $x_J$ represents
an element in the quantum Weyl algebra $\Lambda$. Since the subset $J$ has no repetition,
no confusion will arise in general. When we consider products of $x_{\al}$'s,
we will use multi-components as superscripts to distinguish the elements.

\begin{theorem} \label{T:rel} The following two statements are equivalent.

(i) $x_j^{(k)} x_i^{(k)}=-qx_i^{(k)} x_j^{(k)}$ for $i<j$, $x_i^{(k)}  x_i^{(k)}=0$.

(ii) For any fixed $J=I_1\times \cdots \times I_{m-1}\subset [1, n]^{m-1}$ with $|I_i|=2$ we have that
\begin{align}\label{re a1}
&\sum_{\alpha\sqcup\beta=J}(-q)^{inv(\alpha, \beta)}a_{i\alpha}^{(k)}a_{i\beta}^{(k)}=0,\\ \label{re a2}
&\sum_{\alpha\sqcup\beta=J}(-q)^{inv(\alpha, \beta)}a_{j\alpha}^{(k)}a_{i\beta}^{(k)}=-q\sum_{\alpha\sqcup\beta=J}(-q)^{inv(\alpha, \beta)}a_{i\alpha}^{(k)}
a_{j\beta}^{(k)},
\end{align}
where $inv(\alpha, \beta)=|\{(i_s>j_s)| \alpha=(i_1\ldots i_{m-1}), \beta=(j_1\ldots j_{m-1})\}|$, the total number of inversions from the natural ordering of $J$.
\end{theorem}

\begin{proof} Note that
\begin{equation*}
x_i^{(k)}x_j^{(k)}=\sum_{\alpha, \beta}a_{i\alpha}^{(k)}a_{j\beta}^{(k)}x_{\alpha}x_{\beta}=\sum_J(\sum_{\alpha\sqcup\beta=J}(-q)^{inv(\alpha, \beta)}a_{i\alpha}^{(k)}a_{j\beta}^{(k)})x_J
\end{equation*}
where $J=\{i_1, j_1\}\sqcup\ldots \sqcup \{i_{m-1}, j_{m-1}\}$ is written in the natural
lexicographic order. Then the two relations are obtained by comparing coefficients.
\end{proof}

Let $\mathscr {A}=\mathscr A_q^{[m]}$ be the associative algebra
generated by $a_{i_1\ldots i_{m}}$, $1\leq i_k\leq n$, subject to the relations (\ref{re a1})--(\ref{re a2}).
We arrange the generators into a hypermatrix $A=(a_{i_1\ldots i_{m}})$ of format $n\times\cdots \times n=n^{m}$,
and often call $\mathscr A$ a quantum hypermatrix for simplicity.

When $q^2\neq-1$, by subtracting two relations in Theorem \ref{T:rel} (ii) we obtain that for a fixed product I of 2-element subsets in $[1, n]^{m-1}$ and any $s, t$
and $i<j$, $k<l$
\begin{align}\label{re a3}
\sum_{\alpha\sqcup\beta=J}(-q)^{inv(\alpha, \beta)}a_{i\alpha}^{(s)}a_{j\beta}^{(s)}=\sum_{\alpha'\sqcup\beta'=J}(-q)^{inv(\alpha', \beta')}a_{k\alpha'}^{(t)}
a_{l\beta'}^{(t)}.
\end{align}

Under the assumption that $q^2\neq -1$,  using (\ref{re a1})--(\ref{re a3}) we can derive that for any $k, l$
\begin{equation}\label{re det}
x_1^{(k)}x_2^{(k)}\cdots x_n^{(k)}=x_1^{(l)}x_2^{(l)}\cdots x_n^{(l)}.
\end{equation}

From now on we impose no condition on $q$ and consider arbitrary $q$.
We list two simple properties of the quantum algebra
$\mathscr A$. First, it is straightforward to verify that $f^{(k)}:\Lambda\rightarrow (\mathscr{A})\otimes\Lambda^{\otimes m-1}$ given by
$x_i\mapsto x_i^{(k)}$ defines an algebra homomorphism.

Second, the symmetric group $S_{m}$ acts on the algebra $\mathscr A$ via
\begin{equation}
\sigma a_{i_1,i_2,\ldots,i_{m}}=a_{i_{\sigma^{-1}(1)},i_{\sigma^{-1}(2)},\ldots,i_{\sigma^{-1}(m)}}.
\end{equation}
This can be easily seen as  $(\sigma\sigma') a_{i_1,i_2,\ldots,i_{m}}=\sigma(\sigma' a_{i_1,i_2,\ldots,i_{m}}).$
We denote the homomorphism by $\Theta:S_{m}\longrightarrow Aut(\mathscr{A})$. This property generalizes the duality of rows and
columns in the usual quantum monoid $\Mat_q(n)$.

\section{Quantum hyperdeterminants} 

\subsection{Quantum determinant} We recall how the quantum determinant ${\det}_q(A)$ is defined via a quantum volume element. Let $y_i$ be
defined as in Prop. \ref{P:quadr}, then
\begin{equation}
y_1\wedge \cdots\wedge y_n={\det}_q(A) x_1\wedge \cdots\wedge x_n.
\end{equation}

Using the quantum exterior relations \cite{JZ}, $\det_q$ is explicitly given by
\begin{align}
{\det}_q(A)&=\sum_{\sigma\in S_n}(-q)^{\ell(\sigma)}a_{1\sigma(1)}a_{2\sigma(2)}\cdots a_{n\sigma(n)}\\
&=\sum_{\sigma\in S_n}(-q)^{\ell(\sigma)}a_{\sigma(1)1}a_{\sigma(2)2}\cdots a_{\sigma(n)n}.
\end{align}
where $\ell(\sigma)$ is the number of inversions of the permutation $\sigma$. Usually the two expressions are respectively
called the row determinant and column determinant.

Let $B$ be another $n\times n$-quantum matrix such that its entries commute with those of $A$. Then
\begin{equation}\label{e:proddet}
{\det}_q(AB)={\det}_q(A){\det}_q(B).
\end{equation}
We remark that this holds even if $B$ is an ordinary permutation matrix,
i.e., (\ref{e:proddet}) contains the property that ${\det}_q(A)$ becomes $(-q){\det}_q(A)$ when two rows (column)
are interchanged.

\subsection{Quantum hyperdeterminants}

We now define the quantum analog of Cayley's first hyperdeterminant by (when $q$ is generic)
\begin{equation}\label{e:hdet}
{\Det}_q^{[m]}(\mathscr A)=\frac{1}{[n]_{q^2}!}
\sum_{\sigma_{1},\ldots,\sigma_{m}\in S_{n}}(-q)^{\sum_{i=1}^{m}\ell(\sigma_i)}\prod_{i=1}^{n}
a_{\sigma_1(i),\sigma_2(i),\ldots,\sigma_{m}(i)}.
\end{equation}
The quantum hyperdeterminant is invariant under the action of $S_{m}$:
\begin{equation}
\sigma {\Det}_q^{[m]}(\mathscr A)={\Det}_q^{[m]}(\mathscr A)
\end{equation}
for any $\sigma\in S_{m}$. This invariance generalizes the usual property that $\det_q(A)=\det_q(A^T)$.
Sometimes we will also use $\Det_q(A)$ to denote the quantum hyperdeterminant if there is no confusion.

\begin{example} The hyperdeterminant $A=(a_{ijk})$ of format $2^3$ is given by
\begin{equation}\begin{aligned}
{\Det}_q&=\frac1{[2]_{q^2}}(a_{111}a_{222}-qa_{211}a_{122}-qa_{121}a_{212}-qa_{112}a_{221}\\
&+q^2a_{221}a_{112}+q^2a_{212}a_{121}+q^2a_{122}a_{211}-q^3a_{222}a_{111})
\end{aligned}
\end{equation}
which is not $0$ at $q=1$ in general. However, when all entries $a_{ijk}$ commute with each other,
one has that
\begin{equation}\begin{aligned}
{\Det}_q&=\frac{(1-q)}{[2]_{q^2}}([3]a_{111}a_{222}-qa_{211}a_{122}-qa_{121}a_{212}-qa_{112}a_{221}).
\end{aligned}
\end{equation}
Therefore ${\Det}_1=0$. In fact, when all entries are commutative, the classical
Cayley's first hyperdeterminant vanishes at odd dimension, which was perhaps why Cayley only
defined the hyperdeterminant for hypermatrices at even dimension.
\end{example}

Consider now the algebra $A^{\otimes 2m}$. Let $\varphi$ be the linear map $:\mathscr{A}\longrightarrow A^{\otimes 2m}$ defined by
\begin{equation}
\varphi(a_{i_1,i_2,\ldots,i_{2m}})=a_{i_1,i_2}\otimes a_{i_3,i_4}\otimes\cdots\otimes a_{i_{2m-1},i_{2m}}.
\end{equation}
Then the following result can be easily seen.

\begin{proposition} The map $\varphi$ is an algebra homomorphism and
\begin{equation}
\varphi {\Det}_q^{[2m]}(\mathscr A)= ([n]_{q^2}!)^{2m-1}{\det}_q(A)^{\otimes 2m}.
\end{equation}
Moreover for any $\sigma\in S_{m}$ one has that $\varphi\sigma\in Hom(\mathscr{A},A^{\otimes 2m})$.
\end{proposition}

For a fixed $k\in[1, m]$, let $\eta^{(k)}_i=\sum_{\alpha} a^{(k)}_{i\alpha} x_i\otimes x_{\alpha}=x_i\otimes x_i^{(k)}$,
then $\eta^{(k)}_j\eta^{(k)}_i=q^2\eta^{(k)}_i\eta^{(k)}_j$ for $i<j$ by using Theorem \ref{T:rel}.
Consider $\Omega_k=\sum_{i=1}^{n}\eta^{(k)}_i$, we have
\begin{equation}
\wedge^n\Omega_k=[n]_{q^2}!\eta^{(k)}_1\wedge \eta^{(k)}_2\wedge \cdots \wedge \eta^{(k)}_n,
\end{equation}
where $[n]_{q^2}!={\sum_{\sigma\in S_n}q^{2{\ell(\sigma)}}}$.

Writing out the exterior product, we have that
\begin{equation}\wedge^n \Omega_k=[n]_{q^2}!{\Det}_q^{[m]}(\mathscr{A})(x_1\wedge\cdots\wedge x_n) ^{\otimes m}.
\end{equation}
Therefore
\begin{equation}\eta^{(k)}_1\wedge \eta^{(k)}_2\wedge \cdots \wedge \eta^{(k)}_n
={\Det}_q^{[m]}(\mathscr{A})(x_1\wedge\cdots\wedge x_n) ^{\otimes m}.
\end{equation}

On the other hand, an explicit computation gives that
\begin{align}\nonumber
&\eta^{(k)}_1\wedge \cdots \wedge \eta^{(k)}_n\\
=&\sum_{\sigma_{i}\in S_{n}, i\neq k}(-q)^{\sum_{i\neq k}^{n}\ell(\sigma_i)}
\prod_{i=1}^{m} a_{\sigma_1(i)\ldots\hat{\sigma}_{k}(i)\ldots\sigma_{m}(i)}(x_1\wedge\cdots\wedge x_n) ^{\otimes m},
\end{align}
where $\hat{\sigma}_{k}$ means the omission of $\sigma_k$ at the place (or taking $\sigma_k=1$).
So we also have for any fixed $k$
\begin{equation}
{\Det}_q^{[m]}(\mathscr{A})=\sum_{\sigma_{i}\in S_{n}, i\neq k}(-q)^{\sum_{i\neq k}^{m}\ell(\sigma_i)}\prod_{j=1}^{m} a_{\sigma_1(j) \ldots\hat{\sigma}_{k}(j)\ldots\sigma_{m}(j)},
\end{equation}
which can be used to define the quantum hyperdeterminant for any $q$.

Let $I_1,I_2,\ldots,I_{m}$ be $m$ subsets of $[1, n] $ with $|I_k|=r$. The quantum $r$-minor hyperdeterminants are defined as
\begin{equation}\label{e:minor}
\xi(I_1,I_2,\ldots,I_{m})
=\sum_{(\sigma_2,\ldots,\sigma_{m})\in S_{r}^{m-1}}(-q)^{\sum_{i=2}^{m}\ell(\sigma_i)}
\prod_{i=1}^{r}
a_{i,\sigma_2(i),\ldots,\sigma_{m}(i)},
\end{equation}
where $S_{r}^{m-1}=S_r\times\cdots\times S_r\leq S_{\{r+1, \ldots, 2r, \ldots, (m-1)r+1£¬ \ldots, rm\}}$, the canonical Young subgroup,
i.e., the $j$th factor $S_r$ consists of permutations of letters $(j-1)r+1, \ldots, jr$.

For $m, l$ we can also write the generators of $\mathscr A^{[m+l]}$ as $a_{\alpha\beta}$, where
$\alpha\in[1, n]^m$ and $\beta\in[1, n]^l$. More generally we can
use any composition of $m$ to parameterize the generators
of $\mathscr A$. In particular matrix realignments of $\mathscr A^{[m]}$ are
such examples. The following result is proved by direct
computation.

\begin{theorem}The following map is an algebra homomorphism
\begin{align}
&\mathscr A^{[m+l]}\overset{\Delta}{\longrightarrow} \mathscr A^{[m+1]}\otimes \mathscr A^{[1+l]},\\
& \Delta(a_{\alpha\beta})=\sum_{j=1}^{n}a_{\alpha j}\otimes a_{j\beta}.
\end{align}

Moreover, we have an analogous Laplace expansion:
\begin{equation}
\Delta(\xi(I\times J))=\sum_K\xi(I\times K)\otimes \xi(K\times J).
\end{equation}

In particular,
\begin{equation}
\Delta(\Det_q(\mathscr A_{m+l}))=\Det_q(\mathscr A_{m+1})\otimes \Det_q(\mathscr A_{l+1}).
\end{equation}
\end{theorem}

\begin{proof} For fixed $i$,
write the running indices $\alpha=\alpha_1\alpha_2$ and $\beta=\beta_1\beta_2$, then
\begin{align*}
&\sum_{\alpha,\beta}(-q)^{inv(\alpha, \beta)}\Delta a_{i\alpha} \Delta a _{i\beta}\\
=&\sum_{\alpha,\beta,j,j'}(-q)^{inv(\alpha, \beta)}(a_{i\alpha_1 j} \otimes a_{j\alpha_2 } )
(a _{i\beta_1j'}\otimes a^{(k)}_{j'\beta_2})\\
=&\sum_{\alpha,\beta,j,j'}(-q)^{inv(\alpha_1, \beta_1)+inv(\alpha_2, \beta_2)+inv(j,j')}
a_{i\alpha_1 j}^{(k)}a _{i\beta_1j'} \otimes a_{j\alpha_2 } a _{j'\beta_2}\\
\end{align*}
If $j=j'$, by definition $\sum_{\alpha_2,\beta_2}(-q)^{inv(\alpha_2, \beta_2)}
a_{j\alpha_2 }a_{j'\beta_2}=0$.
If $j\neq j'$, one has that $\sum_{\alpha_1,\beta_1,j\neq j'}(-q)^{inv(\alpha_1, \beta_1)+inv(j,j')}
a_{i\alpha_1 j}a_{i\beta_1j'}=0$.
Thus for each fixed $i$
\begin{equation}
\sum_{\alpha,\beta}(-q)^{inv(\alpha, \beta)}\Delta a_{i\alpha} \Delta a _{i\beta}=0.
\end{equation}
Similarly one can prove that $\sum_{\alpha,\beta}(-q)^{inv(\alpha, \beta)}\Delta a_{i\alpha}^{(k)}\Delta a^{(k)}_{i\beta}=0$ for any $k$ and
$\sum_{\alpha\sqcup\beta=J}(-q)^{inv(\alpha, \beta)}\Delta a_{j\alpha}^{(k)}\Delta a_{i\beta}^{(k)}=-q\sum_{\alpha\sqcup\beta=J}(-q)^{inv(\alpha, \beta)}\Delta  a_{i\alpha}^{(k)}
 \Delta a_{j\beta}^{(k)}.$
Thus $\Delta$ is an algebra homomorphism. Moreover, we have that
\begin{align}\nonumber
\Delta \xi(I\times J)
&=\frac{1}{[n]_{q^2}!} \sum_{\sigma} (-q)^{\sum_{i=1}^{m}\ell(\sigma_i)}
\prod_{i=1}^{r} \sum_{j=1}^{n}a_{\sigma_1(i),\ldots,\sigma_{m}(i),j}\otimes a_{j,\sigma_m(i),\ldots,\sigma_{m+l1}(i)}\\
&=\frac{1}{[n]_{q^2}!} \sum_{\sigma} (-q)^{\sum_{i=1}^{m}\ell(\sigma_i)}
\prod_{i=1}^{r} a_{\sigma_1(i),\ldots,\sigma_{m}(i),j_r}\otimes \prod_{i=1}^{r}a_{j_r,\sigma_m(i),\ldots,\sigma_{m+l1}(i)}
\end{align}
where the summation runs over all $\sigma=(\sigma_1, \ldots, \sigma_{m+l})\in S_{r}^{m+l}$ and $j_1,\ldots,j_r\in [1,n]$.
For distinct $j_1,\ldots,j_r$, one has that $$\sum_{} (-q)^{\sum_{i=m+1}^{m+l}\ell(\sigma_i)}\prod_{i=1}^{r}a_{j_r,\sigma_{m+1}(i),\ldots,\sigma_{m+l}(i)}
=(-q)^{inv(j_1,\ldots,j_r)}\xi(K\times J), $$
where $K=\{j_1,\ldots,j_r\}$.
Then
\begin{equation}
\begin{split}
&\Delta \xi(I\times J)\\
=&\frac{1}{[n]_{q^2}!} \sum_{} (-q)^{\sum_{i=1}^{m}\ell(\sigma_i)+{inv(j_1,\ldots,j_r)}}
\prod_{i=1}^{r} a_{\sigma_1(i),\ldots,\sigma_{m}(i),j_i}\otimes
\xi(K\times J)\\
=&\sum_{K}\xi(I\times K)\otimes\xi(K\times J),
\end{split}
\end{equation}
where $K$ runs over all subsets $K$ of $[1,n]$ such that $|K|=r$. 

Let $I\times J=[1,n]^{m+l}$, one has that
\begin{equation}
\Delta(\Det_q(\mathscr A_{m+l}))=\Det_q(\mathscr A_{m+1})\otimes \Det_q(\mathscr A_{l+1}).
\end{equation}
\end{proof}

\subsection{Laplace expansions of $q$-hyperdeterminants and Pl\"ucker relations.}
 As in the classical case the quantum hyperdeterminant can also be expanded into a
 sum of quantum hyper-minors and the complement hyper-minors. The Laplace expansion
 is similarly proved by the exterior products. In the following we only state the results for the first folding,
 while it is obvious that there are similarly $m$ ways to expand just like there are two equivalent ways to expand
 in rows and columns for the rectangle matrices.

 The following generalizes a result of \cite{TT} and \cite{KL}.
\begin{proposition}  For any $1\leq j_1, \ldots,j_m\leq n$, one has that
\begin{align}\nonumber
&\sum_{\sigma_{2},\ldots,\sigma_{m}\in S_{n}}(-q)^{\sum_{i=2}^{m}\ell(\sigma_i)}\prod_{i=1}^{m} a_{j_i,\sigma_2(i),\ldots,\sigma_{m}(i)}\\
&\quad\qquad =\left\{ \begin{aligned}
&0,\ & \mbox{if two $j_k$'s\ coinside},\\
&(-q)^{\ell(\pi)}{\Det}_q^{[m]}(\mathscr{A})\ &\mbox{if $j_k$'s are distinct},\\
\end{aligned} \right.
\end{align}
where $\pi=\begin{pmatrix} 1 &2 &\ldots &m\\
j_1 &j_2 &\cdots &j_m\end{pmatrix}$.
\end{proposition}
\begin{proof} In the first folding, we simply write $x^{(1)}_j=\omega_j$. It is clear that $\omega_{j_1} \wedge \omega_{j_2}\wedge\cdots \wedge \omega_{j_n}=0$
 whenever two indices coincide.
 For any permutation $\pi$ one has that
\begin{equation}
\omega_{{\pi}_1} \wedge \omega_{\pi_2}\wedge \cdots \wedge \omega_{\pi_n}=(-q)^{\ell(\pi)}\omega_{1} \wedge \omega_{2}\wedge \cdots \wedge \omega_{n}.
\end{equation}
For any composition $(j_1j_2\ldots j_n)$ we compute that
\begin{align*}
&\omega_{{j}_1} \wedge \omega_{j_2}\wedge \cdots \wedge \omega_{j_n}\\
&=(-q)^{\ell(\pi)}\sum_{\sigma_{2},\ldots,\sigma_{m}\in S_{n}}(-q)^{\sum_{i=2}^{m}\ell(\sigma_i)}\prod_{i=1}^{m} a_{j_i,\sigma_2(i),\ldots,\sigma_{m}(i)}(x_1\wedge\cdots\wedge x_n)^{\otimes m-1}.
\end{align*}
So the proposition is proved.
\end{proof}

Now we discuss the Laplace expansion of the quantum hyperdeterminant. We first choose r indices $i_1<i_2<\cdots<i_r$ from $1,2,\ldots,n$ and
let the remaining ones be ${i_{r+1} < i_{r+2} < \cdots < i_n}$. We have
\begin{equation}
\omega_{i_1} \wedge \omega_{i_2}\wedge \cdots \wedge \omega_{i_n}=(-q)^{(i_1+\cdots+i_r-\frac{r(r+1)}{2})}\omega_1 \wedge \omega_2\wedge \cdots \wedge \omega_n
\end{equation}
since $\omega_j \wedge \omega_i=-q\omega_i \wedge \omega_j$, if $i<j$. For any $r$-element subset $J$, we define $\ell(J)=\sum_{i\in I}i-|J|(|J|+1)/2$.

Now let $I_1, \ldots, I_{m}$ be $r$-element subsets of $[1, n]$, then $I_i'=[1, n]-I_i$ are $(n-r)$-element subsets. We compute that

\begin {equation}
\begin{split}
&\omega_{i_1} \wedge \omega_{i_2}\wedge \cdots \wedge \omega_{i_n}\\
=&(\omega_{i_1}\wedge \cdots \wedge w_{i_r})\wedge(w_{i_{r+1}}\wedge \cdots \wedge \omega_{i_n})\\
=&\sum_{I_2,\ldots,I_{m}\in P_n}\xi({I_1,\ldots I_{m}})x_{I_1}\otimes\cdots\otimes x_{I_r}\wedge
\xi({I_{1}'\ldots I_{m}'})x_{I_1'}\otimes\cdots\otimes x_{I_{m}'}\\
=&\sum_{I_2,\ldots,I_{m}\in P_n}(-q)^{\sum_{j=2}^{m}\ell(I_j)}\xi({I_1,\ldots I_{m}})\xi({I_{1}'\ldots I_{m}'})(x_1\wedge\cdots\wedge x_n)^{\otimes m-1}\\
\end{split}
\end{equation}
By comparing two equations it follows that
\begin{proposition} For any $r$-element subsets $I_1, \ldots, I_{m}$ of $[1, n]$, we have
\begin{equation}
{\Det}_q^{[m]}(\mathscr{A})=\sum_{I_2,\ldots,I_{m}\in P_n}(-q)^{\sum_{j=2}^{m}\ell(I_j)-\ell(I_1)}\xi({I_1,\ldots I_{m}})\xi({I'_{1}\ldots I_{m}'})
\end{equation}
\end{proposition}

Let $K=\{k_1, \ldots, k_n\}$ be an $n$-element subset of $[1, 2n]$ such that
$k_1<k_2<\ldots<k_n$. We define $\sigma_K$ to be the
permutation $k_1\ldots k_nk_1'\ldots k_n'$ of $S_{2n}$, where $K'=[1, 2n]-K=\{k_1', \ldots, k_n'\}$ such that $k_1'<k_2'<\ldots<k_n'$,
i.e. $\sigma_K$ is a shuffle.

\begin{proposition}\label{thp1} Let $K_1$ be an $n$-element subset of $[1, 2n]$ such that its elements
 $i_k=k$ for $1\leq k\leq r< n$, $i_{r+1}<i_{r+2}<\cdots<i_{n}$.
Let $K_2, \ldots, K_{m}$ be arbitrary $n$-element subsets of $[1, 2n]$.   Then
\begin{equation}
\begin{split}
&\sum_{I_k}(-q)^{\sum_{t=2}^{m}\ell(\sigma_t)}\xi(I,K_2,\ldots,K_{m})\xi(I,K_2',\ldots,K_{m}')=0,\\
\end{split}
\end{equation}
\begin{equation}
\begin{split}
&\sum_{I_k}(-q)^{-\sum_{t=2}^{m}\ell(\sigma_t)}\xi(I,K_2',\ldots,I_{m}')\xi(I,K_2,\ldots,K_{m})=0,\\
\end{split}
\end{equation}
where $\sigma_t=\sigma_{K_t}$, the permutation associated with $K_t$.
\end{proposition}

\begin{proof} For any $i, r\in I=[1, n]$, let
\begin{align}
\omega_i&=\sum_{\alpha}a_{i\alpha}x_{\alpha}=\sum_{i_2, \ldots, i_m=1}^n a_{i,i_2,\ldots,i_{m}} x_{i_2}\otimes\cdots\otimes x_{i_{m}},\\
\omega_{ir}&=\sum_{\beta}a_{i\beta}x_{\beta}=\sum_{i_2=1}^r\sum_{i_3, \ldots, i_m=1}^n a_{i,i_2,\ldots,i_{m}}x_{i_2}\otimes\cdots\otimes x_{i_{m}},
\end{align}
and let $\omega_{ir}'=\omega_{i}-\omega_{ir}$. Since the second running index is from $1$ to $r$ ($r<n$), one has that
$\omega_{1r}\wedge\cdots \wedge\omega_{nr}=0$. Also it is easy to see that for any $i<j$
\begin{align}
&\omega_{jr}\wedge\omega_{ir}=-q\omega_{ir}\wedge\omega_{jr}, \qquad \omega_{jr}'\wedge\omega_{ir}'=-q\omega_{ir}'\wedge\omega_{jr}',\\
&\omega_{ir}'\wedge\omega_{jr}=-q\omega_{jr}\wedge\omega_{ir}',\\
&\omega_{i}\wedge\omega_{jr}=-\omega_{jr}\wedge(q^{-1}\omega_{ir}+q\omega_{ir}').
\end{align}
Now for $r<n$, we have that
\begin{equation}\label{p11}
\begin{split}
&\omega_1\wedge\omega_2\wedge\cdots\wedge\omega_n\wedge\omega_{1r}'\wedge\cdots\wedge\omega_{nr}'\\
=&(\omega_{1r}+\omega_{1r}')\wedge\cdots\wedge (\omega_{nr}+\omega_{nr}')\wedge\omega_{1r}'\wedge\cdots\wedge\omega_{nr}'\\
=&\omega_{1r}\wedge\cdots\wedge\omega_{nr}\wedge\omega_{1r}'\wedge\cdots\wedge\omega_{nr}'=0
\end{split}
\end{equation}

On the other hand,
\begin{equation}\label{p13}
\begin{split}
&\omega_1\wedge\omega_2\wedge\cdots\wedge\omega_n\wedge\omega_{1r}'\wedge\cdots\wedge\omega_{nr}'\\
=&\sum_{K_t}(-q)^{\sum_{t=2}^{m}\ell(\sigma_t)}\xi(I,K_2,\ldots,K_{m})\xi(I,K_2',\ldots,K_{m}')(x_1\wedge\cdots\wedge x_{2n})^{\otimes (m-1)}\\
\end{split}
\end{equation}
Comparing (\ref{p11}) and (\ref{p13}), we prove the first equation. The second equation follows similarly from (\ref{p13}).
\end{proof}

\begin{proposition}\label{thp3}
Under the same hypothesis of Prop. \ref{thp1}, one has that

\begin{equation}
\begin{split}
&\sum_{K_t}(-q)^{\sum_{t=2}^{m}\ell(\sigma_t)}\xi(I',K_2,\ldots,K_{m})\xi(I,K_2',\ldots,K_{m}')\\
=&(-q)^{n^2-2nr}\sum_{K_t}(-q)^{\sum_{t=2}^{m}n^2-\ell(\sigma_t)}\xi(I,K_2',\ldots,K_{m}')\xi(I',K_2,\ldots,K_{m}).\\
\end{split}
\end{equation}
\end{proposition}

\begin{proof}
It is clear that for any $i<j$
$$\omega_j\wedge\omega_{ir}'=\omega_{ir}'\wedge((-q)^{-1}\omega_{jr}+(-q)\omega_{jr}').$$
So we have
\begin{equation}\label{p21}
\begin{split}
&\omega_{n+1}\wedge\cdots\wedge \omega_{2n}\wedge \omega_{1r}'\wedge \omega_{2r}'\wedge\cdots\wedge \omega_{nr}'\\
=&\omega_{1r}'\wedge \omega_{2r}'\wedge\cdots\wedge \omega_{nr}'\wedge((-q)^{-n}\omega_{n+1,r}+(-q)^{n}\omega_{n+1,r}')\\
&\qquad \qquad \wedge\cdots\wedge((-q)^{-n}\omega_{2n,r}+(-q)^{n}\omega_{2n,r}')\\
=&(-q)^{n^2}\omega_{1r}'\wedge \omega_{2r}'\wedge\cdots\wedge \omega_{nr}'\wedge((-q)^{-2n}\omega_{n+1,r}+\omega_{n+1,r}')\wedge\cdots \\ &\qquad\qquad \wedge((-q)^{-2n}\omega_{2n,r}+\omega_{2n,r}')\\
=&(-q)^{n^2-2nr}\omega_{1r}'\wedge \omega_{2r}'\wedge\cdots\wedge \omega_{nr}'\wedge \omega_{n+1}\wedge\cdots\wedge \omega_{2n}.
\end{split}
\end{equation}
On the other hand, we can expand the wedge product of $\omega_i$'s and $\omega'_j$'s as follows.
\begin{equation}\label{p22}
\begin{split}
&\omega_{n+1}\wedge\cdots\wedge \omega_{2n}\wedge \omega_1'\wedge \omega_2'\wedge\cdots\wedge \omega_n'\\
=&\sum_{K_t}(-q)^{\sum_{t=2}^{m}\ell(\sigma_t)}\xi(I',K_2,\ldots,K_{m})\xi(I,K_2',\ldots,K_{m}')(x_1\wedge\cdots\wedge x_{2n})^{\otimes m-1},
\end{split}
\end{equation}
\begin{equation}\label{p23}
\begin{split}
&\omega_1'\wedge \omega_2'\wedge\cdots\wedge \omega_n'\wedge \omega_{n+1}\wedge\cdots\wedge \omega_{2n}\\
=&\sum_{K_t}(-q)^{\sum_{t=2}^{m}n^2-\ell(\sigma_t)}\xi(I,K_2',\ldots,K_{m}')\xi(I',K_2,\ldots,K_{m})(x_1\wedge\cdots\wedge x_{2n})^{\otimes m-1},
\end{split}
\end{equation}
Comparing (\ref{p21})--(\ref{p23}), we obtain the proposition.
\end{proof}

\subsection{Coaction of $A$ on $\mathscr{A}$.}
We can directly verify that $\mathscr{A}$ has a left $A$-comodule structure given by
\begin{equation}L_{G}^{(k)}(a_{i\alpha}^{(k)})=\sum_{j=1}^{n}a_{ij}\otimes a_{j\alpha}^{(k)}.
\end{equation}
Moreover $L_A$ is an algebra homomorphism, so $\mathscr{A}$ is an $A$-comodule-algebra. Similarly, $\mathscr{A}$ has a right
$A$-comodule-algebra structure given by
\begin{equation}
R_G^{(k)}(a_{i\alpha }^{(k)})=\sum_{j=1}^{n} a_{j\alpha }^{(k)}\otimes a_{ji}.
\end{equation}
Here the quantum algebra $\Mat_q(n)$ is endowed with the comultiplication
$\Delta(a_{ij})=\sum_la_{il}\otimes a_{lj}$.
The following result follows easily by definition.
\begin{proposition} For any $\sigma\in S_m$
\begin{align}
(1\otimes \sigma)L_G^{(k)}\sigma^{-1}&=L_G^{(\sigma(k))},\\
(\sigma \otimes 1)R_G^{(k)}\sigma^{-1}&=R_G^{(\sigma(k))}.
\end{align}
\end{proposition}

Due to this symmetry, we will be focused on $k=1$ for $L_G$ and $k=m$ for $R_G$ in the remaining
part.  If $\xi({I_{1}\ldots I_{m}})$ is an $r$-minor hyperdeterminant, it
is easy to verify that
\begin{align}
L_G(\xi({I_{1}\ldots I_{m}}))=\sum_{|J|=r}\xi^{I_1}_{J}\otimes \xi(J,{I_{2},\ldots, I_{m}}),\\
R_G(\xi({I_{1}\ldots I_{m}}))=\sum_{|J|=r} \xi(I_1,\ldots,
I_{m-1},J)\otimes \xi^{J}_{I_{m}}.
\end{align}
In particular, if $I_1=I_2=\ldots=I_{m}=[1,n]$, we have that
$L_G(\Det_q(\mathscr{A}))={\det}_q(A)\otimes \Det_q(\mathscr{A})$, and
$R_G(\Det_q(\mathscr{A}))= \Det_q(\mathscr{A}) \otimes {\det}_q(A)$.

We recall the connection between the quantum group $\GL_q(n)$ and
the quantum universal enveloping algebra $U_q(gl(n))$ or
rather $U_q(sl(n))$ \cite{J, FRT, NYM, JR}. Let $P$ be the free $\mathbb Z$-lattice of rank n with the
canonical basis $\{\varepsilon_{1},\ldots,\varepsilon_{n}\}$, i.e.
$P=\bigoplus_{i=1}^{n}\mathbb Z\varepsilon_{i}$, endowed with  the symmetric
bilinear form
$\langle\varepsilon_{i},\varepsilon_{j}\rangle=\delta_{ij}$. Now we
define $U_{q}(\mathfrak g)$ as the associative algebra with generators $e_{i},f_{i}(1\leq i
\leq n)$ and $q^{\lambda}(\lambda\in\frac{1}{2}P)$ with the
following relations:
\begin{align}
q^{0}&=1, q^{\lambda}q^{\mu}=q^{\lambda+\mu} \quad (\lambda, \mu \in \frac{1}{2}P),\\
q^{\lambda}e_{k}q^{-\lambda}&=q^{\langle\lambda,\varepsilon_{k}-\varepsilon_{k+1}\rangle}e_{k} \quad (\lambda  \in \frac{1}{2}P, 1\leq k \leq n),\\
q^{\lambda}f_{k}q^{-\lambda}&=q^{-\langle\lambda,\varepsilon_{k}-\varepsilon_{k+1}\rangle}f_{k} \quad (\lambda  \in \frac{1}{2}P, 1\leq k \leq n),\\
e_{i}f_{j}-f_{j}e_{i}&=\delta_{ij}\frac{q^{\varepsilon_{i}-\varepsilon_{i+1}}-q^{-\varepsilon_{i}+\varepsilon_{i+1}}}{q-q^{-1}} \quad (1\leq i, j< n),\\
e_{i}^2e_{j}-&(q+q^{-1})e_ie_je_i+e_je_{i}^2=0  \quad (|i-j|=1),\\
f_{i}^2f_{j}-&(q+q^{-1})f_if_jf_i+f_jf_{i}^2=0  \quad (|i-j|=1),\\
e_{i}e_j&=e_je_i, f_if_j=f_jf_i  \quad (|i-j|>1).
\end{align}
This algebra also has a structure of Hopf algebra with the
following coproduct $\Delta$, counit $\varepsilon$, and antipode
S:
\begin{align}
\Delta(q^\lambda)&=q^\lambda\otimes q^\lambda, \varepsilon(q^\lambda)=1, S(q^\lambda)=q^{-\lambda},\\
\Delta(e_k)&=e_k\otimes
q^{-(\varepsilon_{k}-\varepsilon_{k+1})/2}+q^{-(\varepsilon_{k}-\varepsilon_{k+1})/2}\otimes
e_k,\\
\varepsilon(e_k)&=0, S(e_k)=q^{-1}e_k,\\
\Delta(f_k)&=f_k\otimes
q^{-(\varepsilon_{k}-\varepsilon_{k+1})/2}+q^{-(\varepsilon_{k}-\varepsilon_{k+1})/2}\otimes
f_k,\\ \varepsilon(f_k)&=0, S(f_k)=-qf_k.
\end{align}

For any fixed k, there is a unique pairing of Hopf algebras
\begin{equation}
( \ \ ,\ \ ):U_q(\mathfrak g)\times \GL_q(n)\rightarrow \mathbb C
\end{equation}
satisfying the following relations
\begin{align}
q^{\lambda}(a_{ij})&=\delta_{ij}q^{\langle\lambda,\varepsilon_{i}\rangle},\\
e_{k}(a_{ij})&=\delta_{ik}\delta_{j,k+1},f_{k}(a_{ij})=\delta_{i,k+1}\delta_{j,k},\\
q^{\lambda}(det_q(A)^t)&=q^{t\langle\lambda,\varepsilon_1+\cdots+\varepsilon_n\rangle}  \quad (t\in \mathbb Z),\\
e_{k}(det_q(A)^t)&=f_{k}(det_q(A)^t)=0  \quad (t\in \mathbb Z)
\end{align}
We can regard the element of $U_q(\mathfrak g)$ as a linear functional on
$\GL_{q}(n)$. If $V$ is a right $\GL_q(n)$-comodule (resp. left
$\GL_q(n)$-comodule) with the structure map $L_G:V\rightarrow
V\otimes \GL_q(n)$ (resp. $L_G:V\rightarrow  \GL_q(n)\otimes
V$), then $V$ has a left (resp. right) module structure over $U_q(\mathfrak g)$
defined by
\begin{equation}
x.v=(id\otimes x)R_G(v) \ \ (resp.\ v.x=(x\otimes id)L_G(v)),
\end{equation}
for all $x\in U_q(\mathfrak g)$ and $v\in V$.

The algebra $\mathscr{A}$ becomes a bimodule for $U_q(\mathfrak g)$. The left action of its generators on $\mathscr{A}$ is described
 as follows:
\begin{align}
q^{\lambda}.a_{\alpha i}=q^{\langle\lambda,\varepsilon_{i}\rangle}a_{\alpha i},\\
e_{k}.a_{\alpha i}=\delta_{i,k+1}a_{\alpha,i-1},\\
f_{k}.a_{\alpha i}=\delta_{ik}a_{\alpha,i+1}.
\end{align}
Similarly the right module action is given by
\begin{align}
a_{i\alpha}.q^{\lambda}=q^{\langle\lambda,\varepsilon_{i}\rangle}a_{i\alpha},\\
a_{i\alpha}.e_{k}=\delta_{ik}a_{i+1,\alpha},\\
a_{i\alpha}.f_{k}=\delta_{i,k+1}a_{i-1,\alpha}.
\end{align}
If $x\in U_q(\mathfrak g)$, $\varphi,\psi\in \mathscr{A}$, and
$\Delta(x)=\sum x_{(1)}\otimes x_{(2)}$ then one has
$x.(\varphi\psi)=\sum(x_{(1)}.\varphi)( x_{(2)}.\psi)$ and
$(\varphi\psi).x=\sum(\varphi.x_{(1)})(\psi. x_{(2)})$. In
particular,
$e_{k}.\Det_{q}(\mathscr{A})=f_{k}.\Det_{q}(\mathscr{A})=0$, and
$q^{\lambda}.\Det_{q}(\mathscr{A})=q^{\langle\lambda,\varepsilon_1+\cdots+\varepsilon_n\rangle}\Det_{q}(\mathscr{A})$.

\section{Quantum hyper-Pfaffians}
The algebra $\mathscr B$ associated with the quantum hyper-antisymmetric matrices
is the associative algebra generated by
$b_{i_1i_2\ldots i_{mk}}$, $1\leq i_1,i_2,\ldots,i_{mk}\leq kn$.

We define the quantum hyper-Pfaffian as follows.

\begin{equation}\label{quantum hyper-Pfaffian}
\begin{aligned}
&\Pf_q^{[k,m]}(\mathscr{B})\\
=&\frac{1}{[n]_{q^{k^2}}!}
\sum_{\sigma}
(-q)^{\ell(\sigma)}\prod_{i=1}^{n}b_{\sigma_1((i-1)k+1)\ldots\sigma_1((i-1)k+k)\ldots \sigma_m((i-1)k+1)\ldots\sigma_m((i-1)k+k)},
\end{aligned}
\end{equation}
where $\sigma=(\sigma_1,\ldots,\sigma_m)\in S_k^m$ and $\ell(\sigma)=\sum_{i=1}^{m}\ell(\sigma_i)$. The sum can also be viewed
as running over the permutations of $S_{km}$
such that $\sigma_j((i-1)k+1)<\cdots<\sigma_j((i-1)k+k)$ for any $1\leq i\leq n,1\leq j\leq m$.

\begin{remark}\label{remark hpf}
When $m=1$, the quantum hyper-Pfaffian is the same as the hyper-Pfaffian studied in \cite{JZ}, while the latter generalized and deformed Barvinok's hyper-Pfaffian \cite{B}, see also \cite{LT}.

When $k=2$, the quantum hyper-Pfaffian is a quantum  analog of Matsumoto's hyper-Pfaffian \cite{M}.

When $k=1$, $\Pf_q^{[1,m]}(\mathscr{B})=\frac{[n]_{q^2}!}{[n]_q!}\det_q(\mathscr{B})$.
\end{remark}

Let $I=I_1\times I_2\times \cdots I_m$ be a subset of $[1,kn]^m$ satisfying $|I_j|=k$ for $1\leq j\leq m$.
$i^j_1<i^j_2<\cdots <i^j_k$ are elements in $I_j$. Denote
$b_I=b_{I_1,\ldots,I_m}=b_{i^{1}_{1}\ldots i^{1}_{k}\ldots i^{m}_{1}\ldots i^{m}_{k}}$.

Let $K=K_1\times K_2\times \cdots K_m$ be a subset of $[1,kn]^m$ satisfying $|K_j|=2k$ for $1\leq j\leq m$. Suppose that
\begin{equation}\label{hyPf}
\begin{aligned}
&(-q)^{k^2}\sum_{i_1^1<j^1_1}(-q)^{inv(I,J)}b_{I}b_{J}\\
&=\sum_{i_1^1<j^1_1}(-q)^{inv(J,I)}b_{J}b_{I},
\end{aligned}
\end{equation}
where $I=I_1\times I_2\times \cdots\times I_m$ is the subset of $K$ such that $|I_j|=k$, $J=J_1\times \cdots\times J_m$, $J_i$ is the complement of $I_i$ in $K_i$, $inv(I,J)=\sum_{t=1}^{m}inv(I_t,J_t)$, $inv(I_t, J_t)=|\{(i_s^t>j_s^t)\}|$, the sum runs over all subsets $I$ of $K$ such that $i_1^1<j^1_1$.


The hyper-Pfaffian can be alternatively defined as follows. Let
\begin{align}\label{def pf2}
\Pf_q'(\mathscr{B})=
&\sum_{I} (-q)^{\sum_{s,t}i_t^s-\frac{k(k+1)m}{2}}  b_{I}\Pf_q'(\mathscr{B}_{I^c}),
\end{align}
where the sum runs over all the subset $I$ of $[1,kn]^m$ satisfying $|I_j|=k$ for $1\leq j\leq m$, $i^1_1$ is the smallest in the set $[1,kn]$. $I^c=I_1^c\times\cdots\times I_m^c$.

Similar to \cite{JZ}, the following two lemmas are obtained by induction on $n$.
\begin{lemma}\label{def pf3}
The Pfaffian $\Pf_q'(\mathscr{B})$ can be expanded as follows:
\begin{equation}\label{quantum hyper-Pfaffian}
\Pf_q'(\mathscr{B})
=\sum_{\sigma}
(-q)^{\ell(\sigma)}\prod_{i=1}^{n}b_{\sigma_1((i-1)k+1)\ldots\sigma_1((i-1)k+k)\ldots \sigma_m((i-1)k+1)\ldots\sigma_m((i-1)k+k)},\\
\end{equation}
where $\sigma=(\sigma_1,\ldots,\sigma_m)$, $\ell(\sigma)=\sum_{i=1}^{m}\ell(\sigma_i)$, the sum runs over the permutations
satisfying $\sigma_j((i-1)k+1)<\ldots<\sigma_j((i-1)k+k)$ for any $1\leq i\leq n,1\leq j\leq m$ and $\sigma_1(1)<\sigma_1(k+1)<\sigma_1(k(n-1)+1)$.
\end{lemma}
\begin{proof}
This is verified by definition and induction.
\end{proof}
\begin{lemma}\label{lemma equiv}
One has that
\begin{align}
\sum_{I} (-q)^{\sum_{s,t}i_t^s-\frac{k(k+1)m}{2}}  b_{I}\Pf_q'(\mathscr{B}_{I^c})
=[n]_{q^{k^2}}\Pf_q'(\mathscr{B})
\end{align}
where the sum runs over all subsets $I$ of $[1,kn]^m$ satisfying $|I_j|=k$ for $1\leq j\leq m$, $I^c=I_1^c\times\cdots\times I_m^c$.
\end{lemma}
\begin{proof}
This is easily verified by Lemma \ref{def pf3} using induction.
\end{proof}

\begin{theorem}\label{Equivalent pf}
If the elements of  $\mathscr B$  satisfy relation \eqref{hyPf} for any subset $K$ of $[1,kn]^m$, $|K_j|=2k$, $1\leq j\leq m$, then the two definitions of  the quantum hyper-Pfaffian are equivalent.
\end{theorem}

\begin{proof}
We prove the statement by induction on $n$.
It follows from definition that
\begin{equation}
\Pf_q^{[k,m]}(\mathscr{B})
=\frac{1}{[n]_{q^{k^2}}!}
\sum_{I,\sigma}
(-q)^{\ell(\sigma)+inv(I,I^c)}b_{I}b_{\sigma},\\
\end{equation}
where $b_{\sigma}=\prod_{i=1}^{n}b_{\sigma_2((i-1)k+1)\ldots\sigma_2((i-1)k+k)\ldots \sigma_m((i-1)k+1)\ldots\sigma_m((i-1)k+k)}$, $\sigma=(\sigma_2,\ldots,\sigma_m)$, $\ell(\sigma)=\sum_{i=2}^{m}\ell(\sigma_i)$, and the sum runs over the permutations
satisfying $\sigma_j((i-1)k+1)<\cdots<\sigma_j((i-1)k+k)$ for any $1\leq i\leq n,1\leq j\leq m$.

By induction one has that \begin{equation}
\sum_{\sigma}
(-q)^{\ell(\sigma)} b_{\sigma}={[n-1]_{q^{k^2}}!}\Pf_q'(\mathscr{B}_{I^c}),
\end{equation}
where $\sigma$ is defined as above.

Then
\begin{equation}
\Pf_q^{[k,m]}(\mathscr{B})
=\frac{1}{[n]_{q^{k^2}}}
\sum_{\sigma}
(-q)^{inv(I,I^c)}b_{I}\Pf_q'(\mathscr{B}_{I^c}),
\end{equation}
thus $\Pf_q^{[k,m]}(\mathscr{B})=\Pf_q'(\mathscr{B})$ by Lemma \ref{lemma equiv}.
\end{proof}

We also have the following Laplace expansion for the hyper-Pfaffian.
\begin{proposition}
For any $0\leq t\leq n$,
\begin{equation}
\Pf^{[k,m]}(\mathscr{B})=\left[
\begin{array}{c}
n\\
t
\end{array}\right]_{q^{k^2}}
\sum_{I} (-q)^{inv(I,I^c)}\Pf^{[k,m]}(\mathscr{B}_{I})\Pf^{[k,m]}(\mathscr{B}_{I^c}),
\end{equation}
where the sum runs over all the subset $I$ of $[1,kn]^m$ satisfying $|I_j|=tk$ for $1\leq j\leq m$, $I^c=I_1^c\times\cdots\times I_m^c$.

\end{proposition}
\begin{proof}
Let $\Omega=\sum_{I}b_Ix_I$, where
$I=I_1\times I_2\times \cdots\times I_m$ runs through all the subsets of $[1,kn]^m$ such that $|I_j|=k$, $x_I=x_{I_1}\otimes\cdots \otimes x_{I_m}$. Then
\begin{equation}\label{e:omega}
\Omega^n=[n]_{q^{k^2}}!\Pf^{[k,m]}(\mathscr{B})x_{[1,2n]^m}.
\end{equation}

One the other hand, $\Omega^n=\Omega^t\wedge\Omega^{n-t}$, and
\begin{align}
\Omega^t&=[t]_{q^{k^2}}!\sum_{I}\Pf^{[k,m]}(\mathscr{B}_I)x_I,\\
\Omega^{n-t}&=[n-t]_{q^{k^2}}!\sum_{J}\Pf^{[k,m]}(\mathscr{B}_J)x_J
\end{align}
where $I=I_1\times I_2\times \cdots I_m$, $|I_{s}|=tk$ for $1\leq s \leq m$,
and $J=J_1\times J_2\times \cdots J_m$,
$|J_{s}|=tk$ for $1\leq s \leq m$. Then
\begin{equation}
\Omega^n=[t]_{q^{k^2}}![n-t]_{q^{k^2}}!\sum_{I}\Pf^{[k,m]}(\mathscr{B}_I)\sum_{J}\Pf^{[k,m]}(\mathscr{B}_J)x_Ix_J.
\end{equation}

Note that $x_Ix_J=0$ except $J=I^c$, and $x_Ix_{I^c}=(-q)^{\ell(I,I^c)}x_{[1,kn]^m}$.
Comparing with (\ref{e:omega}), we conclude that
\begin{equation*}
\Pf(\mathscr{B})=\left[
\begin{array}{c}
n\\
t
\end{array}\right]_{q^{k^2}}
\sum_{I} \Pf^{[k,m]}(\mathscr{B}_{I})\Pf^{[k,m]}(\mathscr{B}_{I^c}).
\end{equation*}
\end{proof}

\begin{theorem}\label{Composition pf}
Suppose $k=pk'$, Then
\begin{equation}\label{e:iden-pf}
\Pf^{[k,m]}(\Pf^{[k',m]}(\mathscr{B}_J))=\frac{[pn]_{q^{k'^2}}!}{([p]_{q^{k'^2}}!)^n[n]_{q^{k^2}}!}\Pf^{[k',m]}(\mathscr{B}),
\end{equation}
where
$J=J_1\times \cdots\times J_m$,
$|J_{s}|=k$ for $1\leq s \leq m$.
\end{theorem}
\begin{proof}
Let ${\Omega_{k'}}=\sum_{I}b_Ix_I$, where
$I=I_1\times \cdots\times I_m$ runs through subsets of $[1,kn]^m$ satisfying $|I_j|=k'$, and
$x_I=x_{I_1}\otimes\cdots \otimes x_{I_m}$. Then
\begin{equation}\label{e:omega}
{\Omega_{k'}}^{pn}=[pn]_{q^{k'^2}}!\Pf^{[k',m]}(\mathscr{B})x_{[1,2n]^m}.
\end{equation}

Clearly ${\Omega_{k'}}^{pn}=({\Omega_{k'}}^{p})^n$, and
\begin{equation}\label{e:omk}
{\Omega_{k'}}^{p}=[p]_{q^{k'^2}}!\sum_{J}\Pf^{[k',m]}(\mathscr{B}_J)x_{J},
\end{equation}
where $J=J_1\times\cdots\times J_m$,
$|J_{s}|=k$ for $1\leq s \leq m$.
Then
\begin{equation}\label{e:omkp}
({\Omega_{k'}}^{p})^n=([p]_{q^{k'^2}}!)^n[n]_{q^{k^2}}!\Pf^{[k,m]}(\Pf^{[k',m]}(\mathscr{B}_J)).
\end{equation}
Comparing \eqref{e:omega} and \eqref{e:omkp}, we see that \eqref{e:iden-pf} holds.
\end{proof}

\section{Relationship between Pfaffians and determinants}
\begin{theorem}\label{Pf eq}
Let $\mathscr{B}=(b_{i_1\ldots i_k})_{1\leq i_1,\ldots,i_k \leq kn}$ be a hyper-matrix, and suppose its entries $b_{i_1\ldots i_{m}}$ commute with any element in the algebra $\mathscr {A}$.
For any $I=I_1\times\cdots\times I_{m-1}$, $|I_{s}|=k$, $1\leq s\leq m-1$, let $\mathscr{C}=(c_I)$ be the hypermatrix with
\begin{equation}
c_{I}=\sum_{J}b_J \xi(J,I_1,\ldots,I_{m-1}),
\end{equation}
where $J$ runs over the subsets of $[1,2n]$ such that $|J|=k$, then one has that
\begin{equation}\label{e:iden-pf-det}
\Pf_{q}^{[k,m-1]}(\mathscr{C})={\Det}_q^{[m]}(\mathscr{A})\Pf_{q}^{[k,1]}(\mathscr{B}).
\end{equation}
\end{theorem}

\begin{proof}
Let $\Omega=\sum c_{I}x_{I}$.
Then
\begin{equation}
\Omega^n=[n]_{q^{k^2}}!\Pf^{[k,m-1]}(\mathscr{C})x_{[1,2n]^{m-1}}.
\end{equation}

Since $\omega_J=\sum_{I}\xi(J,I_1,\ldots,I_{m-1})x_I$, one can write that
$$\Omega=\sum_{I}c_{I}x_I=\sum_{I,J}b_J \xi(J,I_1,\ldots,I_{m-1})x_I=\sum_{J}b_J\omega_J.$$
Then $\Omega^n=[n]_{q^{k^2}}!\Pf^{[k,1]}(\mathscr{B})\omega_{[1,kn]}=[n]_{q^{k^2}}!\Pf^{[k,1]}(\mathscr{B}){\Det}_q^{[m]}(\mathscr{A})x_{[1,2n]^{m-1}}.$
From which \eqref{e:iden-pf-det} is obtained.
\end{proof}

\begin{corollary}
Let $b=(b_{ij})$, $1\leq i, j\leq 2n$ be any matrix, $b_{2i-1,2i}=1$, all other entries are $0$.
For any $I=I_1\times I_2\times \cdots I_{m-1}$, $|I_{s}|=2$, $1\leq s\leq m-1$, let $c_{I}=\sum_{J}b_J \xi(J,I_1,\ldots,I_{m-1})$, the sum runs over the subset $J$ of $[1,2n]$ satisfying $|J|=2$ , then ${\Det}_q^{[m]}(\mathscr{A})=\Pf_{q}^{[2,m-1]}(\mathscr{C})$. Moreover, $\Pf_{q}^{[2,m-1]}(\mathscr{C})$
can be simplified as \eqref{def pf2} and  \eqref{def pf3}.
\end{corollary}

\begin{proof}
Since $\Pf_{q}^{[2,1]}(B)=1$, the first statement follows from Theorem \ref{Pf eq}.
By straightforward computation the entries of $\mathscr{C}$ satisfy the relations in Theorem \ref{Equivalent pf}, then $\Pf_{q}^{[2,m-1]}(\mathscr{C})$ can be simplified.
\end{proof}

\begin{proposition}
Let $\mathscr {A}$ be the algebra generated by $a_{i_1\ldots i_{m}}$, $1\leq i_k\leq kn$. For any $J=J_1\times J_2\times \cdots J_{m}$, $|J_{s}|=k$, $b_{J}=\xi_{J}=\xi(J_1,\ldots,J_m)$. Then
\begin{equation}\label{e:iden-pf-det2}
{\Det}_q^{[m]}(\mathscr{A})
=\frac{([n]_{q^2}!)^{n-1}([k]_{q}!)^n[n]_{q^{k^2}}!}{([n]_q!)^{n-1}[kn]_{q}!}\Pf^{[k,m]}(b_J).
\end{equation}
\end{proposition}

\begin{proof}

Suppose $k'=1$ in Theorem \ref{Composition pf}, one has that
\begin{equation}
\Pf^{[k,m]}(\Pf^{[1,m]}(\mathscr {A}_J))=\frac{[kn]_{q}!}{([k]_{q}!)^n[n]_{q^{k^2}}!}\Pf^{[1,m]}(\mathscr {A}),
\end{equation}
where
$J=J_1\times J_2\times \cdots J_m$,
$|J_{s}|=k$, for $1\leq s \leq m$.
It follows from Remark \ref{remark hpf} that
$\Pf_q^{[1,m]}(\mathscr {A}_J)=\frac{[n]_{q^2}!}{[n]_q!}\xi_J=\frac{[n]_{q^2}!}{[n]_q!}b_J$, $\Pf^{[1,m]}(\mathscr {A})=\frac{[n]_{q^2}!}{[n]_q!}{\Det}_q^{[m]}(\mathscr{A})$. Then
\begin{equation}
(\frac{[n]_{q^2}!}{[n]_q!})^n\Pf^{[k,m]}(b_J)
=\frac{[n]_{q^2}!}{[n]_q!}\frac{[kn]_{q}!}{([k]_{q}!)^n[n]_{q^{k^2}}!}{\Det}_q^{[m]}(\mathscr{A}),
\end{equation}
which implies \eqref{e:iden-pf-det2}.
\end{proof}

\section{Conclusion and discussion}
We have defined the notion of quantum hypermatrices and introduced the
quantum hyperdeterminant and quantum hyper-Pfaffian. The quantum
hyperdeterminant has quantized Cayley's first hyperdeterminant
and provided a quantum invariant for the space of the
quantum matrices. In particular, we have obtained the formula:
\begin{equation}
\hdet_q(B\,_l\circ_k A)={\det}_q(B)\hdet_q(A).
\end{equation}
We also proved that the space $\mathscr A$ is a co-bimodule for the tensor product of the
quantum coordinate ring $\Mat_q(n)$:
\begin{equation}
\mathscr A\longrightarrow \Mat_q(n)^{\otimes m}\otimes \mathscr A\otimes \Mat_q(n)^{\otimes m}.
\end{equation}
Using the dual co-algebra structure, we also obtain that $\mathscr A$ is a bi-module
for $U_q(sl(n))^{\otimes m}$:
\begin{equation}
U_q(sl(n))^{\otimes m}\otimes \mathscr A\otimes U_q(sl(n))^{\otimes m}\longrightarrow \mathscr A.
\end{equation}
Using this map we have shown that the image of $\hdet_q$ is exactly $\det_q^{\otimes m}$
in the space $\GL_q(n)^{\otimes m}$. However, our hyperdeterminant $\hdet_q$ is not
a central element in $\mathscr A$. This means that the algebra $\mathscr A$
is a very general uplift of the algebra  $\GL_q(n)^{\otimes m}$.

We conjecture that there exists an intermediate quotient algebra $\overline{\mathscr A}=\mathscr A/I$,
where the ideal $I$ contains some relations that make the quantum hyperdeterminant $\hdet_q$ a central element.
This ideal is trivial in the case of 2-dimensional quantum matrices (i.e. $m=1$ case). We note that
quantum analogs of Cayley's other hyperdeterminants may provide a solution, as
indicated by the classical case (cf. \cite{R1, R2, GKZ}).
On the other hand, this also means that we have obtained some optimal relations to define both
quantum hyperdeterminant and the quantum hyper-Pfaffian.

For $m=4$, if one uses the quantum Weyl algebra to define
the quadra\-tic relations for quantum hypermatrices,
there are more relations than our current approach of using the quantum
exterior algebra. However, it seems that the quantum Weyl algebra is also not enough to
make $\hdet_q$ central.

\bigskip

\centerline{\bf Acknowledgments}
This work is supported by NSFC grant Nos. 11271138 and 11531004, and Simons Foundation 198129.
The second author thanks the hospitality of North
Carolina State University and support from China Scholarship Council during the project.

 \vskip 0.1in

\bibliographystyle{amsalpha}

\end{document}